\documentclass[leqno]{amsart}
\usepackage{amsmath,amsfonts,amsthm,amssymb,indentfirst,url}

\newcommand{\<}{\kern.0833em}

\newtheorem{theorem}{Theorem}
\newtheorem{lemma}[theorem]{Lemma}

\newtheorem{conjecture}[theorem]{Conjecture}
\newtheorem{proposition}[theorem]{Proposition}
\newtheorem{definition}[theorem]{Definition}
\newtheorem{question}[theorem]{Question}

\newcommand{\ch}{\r{ch}}
\renewcommand{\r}{\mathrm}

\begin{document}

\begin{center}
\texttt{Comments, corrections, and related references welcomed, as
  always!}\\[.5em]
{\TeX}ed \today
\vspace{2em}
\end{center}

\title{On common divisors of multinomial coefficients}%
\thanks{This preprint is readable online at
\url{http://math.berkeley.edu/~gbergman/papers/}
and \url{http://arXiv.org/abs/0806.0607}\,.
The former version is likely to be updated more frequently than
the latter.
}

\author{George M. Bergman}

\begin{abstract}
Erd\H{o}s and Szekeres showed in 1978 that for any four positive
integers satisfying $m_1+m_2=n_1+n_2,$ the two binomial coefficients
$(m_1+m_2)!/m_1!\,m_2!$ and $(n_1+n_2)!/n_1!\,n_2!$ have a common
divisor $>1.$
The analogous statement for families of $k$ $\!k\!$-nomial coefficients
$(k>1)$ was conjectured in 1997 by David Wasserman.

Erd\H{o}s and Szekeres remark that if
$m_1,\ m_2,\ n_1,\ n_2$ as above are all $>1,$ there is probably
a lower bound on the common divisor in question which goes to infinity
as a function of $m_1+m_2.$
Such a bound is obtained in \S\ref{S.bounds}.

The remainder of this note is devoted to proving results that narrow
the class of possible counterexamples to Wasserman's conjecture.
\end{abstract}
\maketitle

Above, I have worded Erd\H{o}s and Szekeres's result so as to make
clear the intended generalization to $k$ $\!k\!$-nomial coefficients.
In the next two sections, however, I formulate it
essentially as they do in~\cite{ESz}.

I have ``trimmed the fat'' from an earlier, lengthier version of
this note.
The material removed can be found in~\cite{from_nomial}.

I am indebted to Pace Nielsen for a number of valuable
corrections and comments.

\section{Background: the result of Erd\H{o}s and Szekeres.}\label{S.ESz}

We begin with two quick proofs of Erd\H{o}s and Szekeres's result, one
roughly as in~\cite{ESz}, the other group-theoretic.

\begin{theorem}[Erd\H{o}s and Szekeres \cite{ESz}]\label{T.ESz}
Suppose $i,\ j,\ N$ are integers with $0<i\leq j\leq N/2.$

Then $\binom{N}{i}$ and $\binom{N}{j}$ have a common divisor $>1.$
\end{theorem}
\noindent{\em Proof following \cite{ESz}.}
Note that
\begin{equation}\begin{minipage}[c]{35pc}\label{d.2ways}
$\binom{N}{i}\ \binom{N-i}{j-i}\ =\ \binom{N}{j}\ \binom{j}{i}.$
\end{minipage}\end{equation}
Now if the first factors on the two sides of the above equation
were relatively prime, the second factor on each side would have
to be divisible by, and hence at least as large as, the first factor on
the other side.
In particular, we would have $\binom{j}{i}\geq\binom{N}{i}.$
Multiplying both sides by $i\,!,$ this would say
$j\<(j{-}1)\dots(j{-}\<i{+}1)\geq N(N{-}1)\dots(N{-}\<i{+}1),$
which is clearly false.\vspace{.5em}

\noindent{\em Group-theoretic proof.}
Let $X$ denote the set of decompositions $(A,B)$ of
$\{1,\dots,N\}$ into a set $A$ of $i$ elements and
a complementary set $B$ of $N{-}\<i$ elements, and
$Y$ the set of decompositions $(C,D)$ of the same
set into complementary sets of $j$ and $N{-}\<j$ elements.
The permutation group $S_N$ acts transitively on each
of these sets, which have cardinalities
$\binom{N}{i}$ and $\binom{N}{j}$ respectively.

Consider the product action of $S_N$ on $X\times Y.$
Each orbit must have cardinality divisible by both
$\r{card}(X)$ and $\r{card}(Y),$ hence if these were
relatively prime, every orbit would have cardinality
at least $\r{card}(X\times Y),$ so there could be only one orbit.
But in fact, the orbits correspond to the possible choices
of cardinalities for $A\cap C,\ B\cap C,\ A\cap D$ and $B\cap D,$
and these can be chosen in different ways; e.g., so that
$A\subseteq C$ or so that $A\subseteq D,$ so there
are at least two orbits.\qed\vspace{6pt}

\section{Lower bounds}\label{S.bounds}

The above two proofs are not as different as they look:
the value~(\ref{d.2ways}) is the trinomial coefficient
$N!/i!\,(j{-}\<i)!\,(N{-}\<j)!,$ which counts the orbit of $X\times Y$
consisting of decompositions with $A\subseteq C.$
(The right-hand side of~(\ref{d.2ways}) essentially says
``break $\{1,\dots,N\}$ into $C$ and $D,$ then choose $A$
within $C$'', while the left-hand side says ``break
$\{1,\dots,N\}$ into $A$ and $B,$ then choose $C-A$ within $B$''.)

Note that in the first proof above, the ratio of the numbers we
compared, $N(N{-}1)\dots(N{-}\<i{+}1)$ and $j(j{-}1)\dots(j{-}\<i{+}1),$
can be written
$(N/j)((N{-}1)/(j{-}1))\dots((N{-}\<i{+}1)/(j{-}\<i{+}1))\geq 2^i.$
As noted in \cite{ESz}, this implies that $\binom{N}{i}$
and $\binom{N}{j}$ have a common divisor $\geq 2^i.$
This estimate goes to infinity with $i,$ but gives no information
on how the greatest common divisor of these
numbers behaves as a function of $N$
for $i$ fixed; indeed, it is observed in~\cite{ESz} that when $i=1,$
that greatest common divisor is $2$ in infinitely many cases.
Let us now show, however, as Erd\H{o}s and Szekeres
suspected, that when
\begin{equation}\begin{minipage}[c]{35pc}\label{d.i>1}
$i\ >\ 1,$
\end{minipage}\end{equation}
that greatest common divisor goes to infinity with $N.$
To this end, we shall bring in the other orbits of our action of $S_N.$

Fixing $N,\ i,$ and $j,$ we find that for each orbit of pairs of
decompositions $\{1,\dots,N\}=A\sqcup B=C\sqcup D,$
the integer $h=\r{card}(A\cap D)$ is an invariant of the
orbit, uniquely determining the orbit, and that this
invariant can take on any value satisfying $0\leq h\leq i.$
The cardinality of the orbit associated with $h$ is
given by the $\!4\!$-nomial coefficient
\begin{equation}\begin{minipage}[c]{35pc}\label{d.orbitsize}
$\binom{N}{i}\binom{i}{h}\binom{N-i}{j-i+h}\ =\ Q_h\ =
\ \binom{N}{j}\binom{j}{i-h}\binom{N-j}{h}.$
\end{minipage}\end{equation}

By~(\ref{d.orbitsize}), each of these values $Q_h$
must be divisible by the integer
\begin{equation}\begin{minipage}[c]{35pc}\label{d.L}
$L\ =\ \r{l.c.m.}\,(\binom{N}{i},\,\binom{N}{j}).$
\end{minipage}\end{equation}
Our idea is that as $h$ varies from $0$ to $i,$ $Q_h$ should vary in
a ``nice'' fashion; but if the above value $L$ were too large,
the big gaps between the available values would make this impossible.
Let us try out this idea on $Q_0,\ Q_1$ and $Q_2.$
Since multinomial coefficients are multiplicative in
nature, let us subtract the product of the first and last of
these from the square of the middle one,
after multiplying these products by integer factors
$(2\<i$ and $i-1$ respectively) that compensate for the
different denominators of the binomial coefficients in question.
Expanding the $Q_h$ by the right-hand expression
in~(\ref{d.orbitsize}), we can say that $L^2$ divides
\begin{equation}\begin{minipage}[c]{35pc}\label{d.A}
$(i\<{-}1)\,Q_1^2\ -\ 2\<i\,Q_0\,Q_2\\[.3em]
\hspace*{1em}=\ \binom{N}{j}^2
\ ((i\<{-}1)\binom{j}{i-1}^2\binom{N-j}{1}^2-
2\<i\,\binom{j}{i}\binom{j}{i-2}\binom{N-j}{0}\binom{N-j}{2})\\[.3em]
\hspace*{1em}=\ \binom{N}{j}^2\binom{j}{i-2}^2
\ ((i\<{-}1)(\frac{j-i+2}{i-1})^2\,(\frac{N-j}{1})^2
-2\<i\,\frac{(j-i+2)(j-i+1)}{i(i\<{-}1)}\,
\frac{(N-j)(N-j-1)}{2\<\cdot 1})\\[.3em]
\hspace*{1em}=\ \frac{(j-i+2)(N-j)}{i-1}\ \binom{N}{j}^2\binom{j}{i-2}^2
\ ((j-i+2)(N-j)-(j-i+1)(N-j-1))\\[.3em]
\hspace*{1em}=\ \frac{(j-i+2)(N-j)}{i-1}\ \binom{N}{j}^2\binom{j}{i-2}^2
\ (N-i+1).$
\end{minipage}\end{equation}

Since the above expression is positive, it gives an upper bond on $L^2.$
Moreover, the cancellation, at the last
step, of the degree-$\!2\!$ terms in the final
parenthesis gives the goal we were aiming for: The above upper bound
is of smaller magnitude than the products we started with.
We now make some estimates to get a simpler expression.
Note that $j\<{-}\<i\<{+}\<2\leq j\leq N/2,$ $N{-}\<j<N,$
$\binom{j}{i-2}\leq \binom{N}{i-2}/2^{i-2}$ (cf.\ second paragraph
of this section), and $N{-}\<i\<{+}1<N.$
Hence~(\ref{d.A}) gives
\begin{equation}\begin{minipage}[c]{35pc}\label{d.L^2leq}
$L^2\ <\ \frac{N^3}{2^{2i-3}(i-1)}\ \binom{N}{j}^2\binom{N}{i-2}^2.$
\end{minipage}\end{equation}

Now the greatest common divisor of $\binom{N}{i}$ and
$\binom{N}{j}$ is their product divided by $L.$
When we divide their product by the square root of the right hand
side of~(\ref{d.L^2leq}),
the factors $\binom{N}{j}$ cancel, while $\binom{N}{i}$ in the
product and $\binom{N}{i-2}$ in the
bound on $L$ almost cancel, with quotient
$(N{-}\<i\<{+}\<2)(N{-}\<i\<{+}1)/i\<(i-1).$
So the g.c.d.\ is at least
\begin{equation}\begin{minipage}[c]{35pc}\label{d.prod/sqt}
$\frac{(N-i+2)(N-i+1)}{i\<(i-1)}\,(\frac{2^{2i-3}(i-1)}{N^3})^{1/2}.$
\end{minipage}\end{equation}

Bounding $N{-}\<i\<{+}\<2$ and
$N{-}\<i\<{+}1$ below by $N/2$ we get the final bound in

\begin{theorem}\label{T.bound}
Suppose $i,\ j,\ N$ are integers with $2\leq i\leq j\leq N/2.$
Then the greatest common divisor of $\binom{N}{i}$ and $\binom{N}{j}$
is bounded below by\textup{~(\ref{d.prod/sqt})}; hence by
\begin{equation}\begin{minipage}[c]{35pc}\label{d.rough}
$N^{1/2}\ 2^{i-7/2}\,/\,i\,(i-1)^{1/2}.$\qed
\end{minipage}\end{equation}
\end{theorem}

For each $i,$~(\ref{d.rough}) goes to infinity as a function of $N;$
clearly it can also be weakened to a bound
that goes to infinity in $N$ independent of $i.$

Can one get still better bounds if one assumes $i>2$?
Our calculation above was based on the idea that $Q_0\,Q_1^{-2}\,Q_2$
should be well-behaved;
note that the exponents in that expression are
the binomial coefficients $1,\ 2,\ 1$ taken with alternating signs; so
for $i\geq 3,$ the expression $Q_0\,Q_1^{-3}\,Q_2^3\,Q_3^{-1}$ might
be still better behaved, suggesting that one look at the difference
between appropriate integer multiples of $Q_0\,Q_2^3$ and $Q_1^3\,Q_2.$
But in fact, the higher power of $L$ that would be involved in the
analog of~(\ref{d.L^2leq}) seems to negate the advantage coming
from the larger number of terms which cancel in that difference.
On the other hand, for $i\geq 4,$ an appropriate linear combination
of $Q_0\,Q_4,\ Q_1\,Q_3$ and $Q_2^2$ might yield an improved estimate
without suffering from this disadvantage.
I leave these questions to others to investigate.

(We remark that the focus of~\cite{ESz} was not the question answered
above, but the value of the largest prime dividing
both $\binom{N}{i}$ and $\binom{N}{j}.)$

\section{Wasserman's conjecture on multinomial
coefficients.}\label{S.defs&}

Before stating and discussing the conjectured generalization
of Theorem~\ref{T.ESz}, let us set up a notation and language for
multinomial coefficients, and record some immediate properties
thereof.

\begin{definition}\label{d.k-nom}
If $a_1,\dots,a_k$ are nonnegative integers, we define
\begin{equation}\begin{minipage}[c]{35pc}\label{d.defch}
$\ch(a_1,\dots,a_k)\ =\ (a_1+\dots+a_k)!\,/\,a_1!\,\dots\,a_k!$
\end{minipage}\end{equation}
\textup{(}modeled on the reading ``$n$-choose-$m\!$''
for binomial coefficients\textup{)}.
Thus, $\ch(a_1,\dots,a_k)$ counts the ways of
partitioning a set of cardinality
$a_1+\dots+a_k$ into a list of subsets, of
respective cardinalities $a_1,\dots,a_k.$

We shall call an integer~\textup{(\ref{d.defch})} a
\textup{$\!k\!$-nomial coefficient} of {\em weight}
$a_1+\dots+a_k.$
It will be called a {\em proper} $\!k\!$-nomial coefficient
if none of the $a_i$ is zero.

A $\!k\!$-nomial coefficient will also be called
a {\em multinomial} coefficient of {\em nomiality} $k.$
\end{definition}

The more usual notation for multinomial coefficients
is $\binom{\,a_1+\,\dots\,+\,a_k\,}{a_1,\ \dots\,,\ a_k};$
but $\ch(a_1,\dots,a_k)$ is visually simpler.

We note three straightforward identities.
First, ``monomial'' coefficients are trivial:
\begin{equation}\begin{minipage}[c]{35pc}\label{d.1nom}
$\ch(n)\ =\ 1.$
\end{minipage}\end{equation}

Second, the operator $\ch$ is commutative, i.e., invariant under
permutation of its arguments:
\begin{equation}\begin{minipage}[c]{35pc}\label{d.com}
$\ch(a_1,\dots,a_k)\ =\ \ch(a_{\pi(1)},\dots,a_{\pi(k)})$
\quad for $\pi\in S_k.$
\end{minipage}\end{equation}

Finally, given any {\em string of strings} of nonnegative integers,
$a_1,\dots,a_{j_1};\ \dots\,;\ a_{j_{k-1}+1},\dots,a_{j_k},$
we have the associativity-like relation
\begin{equation}\begin{minipage}[c]{35pc}\label{d.assoc}
$\ch(a_1,\dots,a_{j_k})\\[.3em]
\hspace*{1em}=\ %
\ch(a_1+\dots+a_{j_1},\ \dots\ ,a_{j_{k-1}+1}+\dots+a_{j_k})\cdot
\ch(a_1,\dots,a_{j_1})\cdot \,\ldots\,
\cdot \ch(a_{j_{k-1}+1},\dots,a_{j_k}).$
\end{minipage}\end{equation}

In particular,~(\ref{d.assoc}) tells us that any multinomial coefficient
is a multiple of any multinomial coefficient (generally of smaller
nomiality) obtained from it by collecting
and summing certain of its arguments.\vspace{.3em}

In the language introduced above, Erd\H{o}s and Szekeres's result
says that any two proper binomial coefficients of equal weight $N$
have a common divisor $>1.$
This implies the same conclusion for two
proper $\!k\!$-nomial coefficients of equal weight $N,$ for any
$k\geq 2,$ since~(\ref{d.assoc}) implies that these are multiples
of two proper binomial coefficients of weight $N.$
But a stronger statement would be

\begin{conjecture}[David Wasserman, personal communication, 1997;
cf.~{\cite[p.131]{RKG}}]\label{Cj}
For every $k>1,$ every family of $k$ proper $\!k\!$-nomial
coefficients of equal weight $N$ has a common divisor $>1.$
\end{conjecture}

Note the need for the condition $k>1:$ by~(\ref{d.1nom}), the
corresponding statement with $k=1$ is false.

It is not clear whether one can somehow adapt the methods of the
preceding sections to prove this conjecture.
Even looking at the case $k=3,$ one finds that the set of
orbits into which the product of three orbits of $\!3\!$-fold
partitions of $\{1,\dots,N\}$ decomposes is an unwieldy structure,
in which the single index ``$\!h\!$'' that parametrized the orbits
in a product of two orbits of $\!2\!$-fold partitions is replaced by
$20$ parameters (computation sketched in~\cite[\S1]{from_nomial}).
Moreover, to get divisors
common to a threesome of trinomial coefficients by studying
the set of their common multiples, one would presumably have
to apply the inclusion-exclusion principle, for which one would
need {\em upper} bounds on their {\em pairwise} common divisors.

In the remaining sections we will take a more pedestrian approach,
and obtain results narrowing the class of cases where one might
look for counterexamples to the above conjecture.

\section{A lemma of Kummer.}\label{S.Kummer}

A basic tool in studying divisibility properties of
multinomial coefficients is the next result, proved
for binomial coefficients, i.e., for $k=2,$ by Kummer.
His proof generalizes without difficulty to arbitrary $k.$
(The result appears in~\cite{EK} as the third-from-last display on
p.116.
The symbol $\mathit{\Pi}(n)$ there denotes what is now written~$n!\,.)$

\begin{lemma}[{after Kummer~\cite{EK}, cf.~\cite{AG}}]\label{L.Kummer}
Let $a_1,\dots,a_k$ be natural numbers, and $p$ a prime.
Then the power to which $p$ divides $\ch(a_1,\dots,a_k)$
is equal to the number of ``carries'' that must be performed
when the sum $a_1+\dots+a_k$ is computed in base~$p.$

In particular, $\ch(a_1,\dots,a_k)$ is relatively prime to $p$
if and only if that sum can be computed ``without carrying'', i.e., if
and only if for each $i,$ the sum of the coefficients of $p^i$ in the
base-$\!p\!$ expressions for $a_1,\dots,a_k$ is less than $p$
\textup{(}and hence gives the coefficient of $p^i$
in the base-$\!p\!$ expression for $a_1+\dots+a_k).$
\qed
\end{lemma}

Remarks: In the classical case $k=2,$ the meaning of the ``number
of carries'' is clear: it is the number of values of $i$
for which the coefficient of $p^i$ in the expression for
$a_1+a_2$ is not the sum of the corresponding
coefficients from $a_1$ and $a_2,$ possibly augmented by
a $1$ carried from the next-lower column, but rather,
the result of subtracting $p$ from that sum.
In the case of $k$ summands $a_1,\dots,a_k,$ we could define the number
of carries recursively in terms of $\!2\!$-term addition, as the
sum of the number of ``carries'' that occur in adding
$a_2$ to $a_1,$ the number that occur in adding $a_3$ to that sum, etc..
Or we could consider
the computation of $a_1+\dots+a_k$ to be performed all at once,
by a process of successively adding up,
for each $i,$ the coefficients of $p^i$ in the summands,
together with any value carried from lower digits, writing the
result as $s_i\<p+t_i$ with $0\leq t_i<p,$ taking
$t_i$ to be the coefficient of $p^i$ in the sum, and ``carrying'' $s_i$
into the next column.
We would then consider this step of the calculation to contribute
$s_i$ to our tally of the number of ``carries''.
Since turning a coefficient $p$ of $p^i$ into
a coefficient $1$ of $p^{i+1}$ reduces the sum of the digits
by $p-1,$ the number of carries under either description can be
evaluated by summing the digits in the base-$\!p\!$ expressions
for $a_1,\dots,a_k,$ subtracting from their total the sum of the
digits of $a_1+\dots+a_k,$ and dividing by $p-1.$

However, we will not need the exact value of the
number of carries, but only to know when it is zero, and for this,
the easy criterion in the second paragraph of the above lemma suffices.

In discussing Conjecture~\ref{Cj}, the
following language will be useful.

\begin{definition}\label{D.accept}
Given a positive integer $N$ and a prime $p,$ we will
call a decomposition of $N$ as a sum of positive integers
$N=a_1+\dots+a_k$ {\em $\!p\!$-acceptable} if
$\ch(a_1,\dots,a_k)$ is not
divisible by $p;$ equivalently, if for each $i,$ the
$\!i\!$-th digit of the base-$\!p\!$ expression for
$N$ is the sum of the $\!i\!$-th digits of the base-$\!p\!$ expressions
for $a_1,\dots,a_k.$
\end{definition}

This language reflects the point of view of someone trying to find a
counterexample to the conjecture: such a counterexample for given $k$
and $N$ would require a set of $k$ decompositions of $N$ into
$k$ positive summands, such that for every prime $p,$ at least
one of these decompositions is $\!p\!$-acceptable.
And, indeed, in obtaining our results supporting the conjecture,
we shall in general put ourselves in the position of trying
to find such a counterexample, and discover obstructions to
getting $\!p\!$-acceptability for all $p.$

Note that by the last criterion in Definition~\ref{D.accept}, we have:
\begin{equation}\begin{minipage}[c]{35pc}\label{d.sum<k}
If an integer $N$ has digit-sum $<k$ when written to
the prime base $p,$ then every proper $\!k\!$-nomial
coefficient of weight $N$ is divisible by $p.$
\end{minipage}\end{equation}

For $k=3,$ the condition of digit-sum $<k$ to base $p$
means that $N$ can be written $p^e$ or $p^e+p^{e'}.$
Examining the first $100$ positive integers, one finds that $75$
of them can be so written for some prime $p,$ and so cannot be the
weight of a counterexample to Wasserman's conjecture for that $k.$
The remaining $25$ are
\begin{equation}\begin{minipage}[c]{35pc}\label{d.25_<100}
$15,\ 21,\ 35,\ 39,\ 45,\ 51,\ 52,\ 55,\ 57,\ 63,\ 69,\ 70,
\ 75,\ 76,\ 77,\ 78,\ 85,\ 87,\ 88,\ 91,\ 92,\ 93,\ 95,\ 99,\ 100.$
\end{minipage}\end{equation}

My early pursuit of this problem involved case-by-case elimination
of these values.
In \cite[\S3]{from_nomial} I reproduce the
ad hoc arguments for one of the less easily eliminated cases, $N=78.$
Below, however, we shall give general arguments that exclude
all values in a much larger range.

Let us note one other immediate consequence of
the final criterion of Definition~\ref{D.accept}.
\begin{equation}\begin{minipage}[c]{35pc}\label{d.2vals}
If $N=a_1+\dots+a_k$ is a $\!2\!$-acceptable decomposition
of a positive integer, then the
powers of $2$ dividing $a_1,\dots,a_k$ are distinct.
\end{minipage}\end{equation}

\section{Preview of results on
Wasserman's Conjecture for $k=3.$}\label{S.k=3}

Suppose we are given a positive integer $N,$ and wish to know whether
it is a counterexample to Conjecture~\ref{Cj} for $k=3;$ i.e., whether
there exist three decompositions of $N$ as sums of three
positive integers, such that for every prime $p,$ one of those
decompositions is $\!p\!$-acceptable.

If so, then one of those
decompositions must have a summand quite close to $N.$
Namely, if $p_{\max}^d$ is the largest prime power $\leq N,$ then
a decomposition that is $\!p_{\max}\!$-acceptable must include
a summand $\geq p_{\max}^d$ (since on adding its
summands in base $p_{\max},$ the digit
in the $p_{\max}^d$ column cannot arise by carrying).
Let us write that decomposition as
\begin{equation}\begin{minipage}[c]{35pc}\label{d.N-i-j}
$N\ =\ (N{-}\<i\<{-}\<j)+i+j,$
\end{minipage}\end{equation}
where $N{-}\<i\<{-}\<j\geq p_{\max}^d,$ so that $i$ and $j$ are small.
Note that the corresponding trinomial coefficient
\begin{equation}\begin{minipage}[c]{35pc}\label{d.chN-i-j}
$\ch(N{-}\<i\<{-}\<j,\,i,\,j)\ =
\ N(N{-}1)\dots(N{-}\<i\<{-}\<j\<{+}1)\,/\,i!\,j!$
\end{minipage}\end{equation}
is not divisible by any prime {\em not} dividing
one of $N,\ N{-}\<1,\,\dots,\,N{-}\<i\<{-}\<j\<{+}\<1;$
so primes not dividing any of those
integers can be ignored in studying the conditions that must be
satisfied by the other two decompositions of~$N.$
On the other hand, $\ch(N{-}\<i\<{-}\<j,\,i,\,j)$
{\em will} tend to be divisible by the primes that do divide one
of~$N,\ N{-}\<1,\,\dots,\,N{-}\<i\<{-}\<j\<{+}\<1:$ that can only
fail if the relatively
small denominator $i!\,j!$ in~(\ref{d.chN-i-j}) cancels all
occurrences of those primes in the numerator.

The further study of this situation bifurcates
into two cases:  If $i$ and $j$ are as small as possible, namely,
both equal to $1,$ then in deducing conditions that must be
satisfied by the other two decompositions,
we have the advantage of knowing that no primes
are cancelled by the denominator $i!\,j!$
in~(\ref{d.chN-i-j}); on the other hand,
the only primes we have to work with are those dividing $N(N{-}\<1).$
We shall study this situation in the next section, and show
that there can be no such example
with $N<1726,$ or, if $N$ is even, with $N<6910.$

In \S\ref{S.i+j>2} we study the reverse situation, where $i+j>2.$
Here the use of the primes
dividing $N(N{-}\<1)(N{-}\<2)$ will prove significantly stronger than
the use of primes dividing $N(N{-}\<1),$ but the cancellation
of factors by the denominator $i!\,j!$ will take its toll.
That problem is not serious for low values of $i+j:$
we will find that there can be no counterexamples with $2<i+j<11.$
Thus, in any counterexample falling under this case
we must have $N-p_{\max}^d\geq 11.$
Looking individually at the first few $N$ satisfying that inequality,
and applying ad hoc considerations to these,
we shall show that there are no counterexamples with $N<785.$

What about the primes dividing
$(N{-}\<3)\,\dots\,(N{-}\<i\<{-}\<j\<{+}1)$?
For any particular $N,$ these can be useful in excluding
possible counterexamples; but the general methods
of \S\ref{S.i+j>2} below cannot make use of them.
Perhaps some reader will succeed in doing so.

\section{The case $i=j=1.$}\label{S.i=j=1}

The proposition below gives the first of the two results
previewed above, and a bit more.
In the proof, and the remainder of this note,
by ``the prime-powers factors of $N$'' we shall mean the factors
occurring in the decomposition of $N$ into positive powers of
{\em distinct} primes.
(So in this usage, $4$ is among the ``prime-power factors''
of $12,$ but $2$ is not.)

\begin{proposition}\label{P.(N-2)+1+1}
Suppose $N$ is a positive integer having decompositions with positive
integer summands,
\begin{equation}\begin{minipage}[c]{35pc}\label{d.aaabbb}
$N\ =\ (N-2)\ +\ 1\ +\ 1,\qquad
N\ =\ a_1\ +\ a_2\ +\ a_3,\qquad N\ =\ b_1\ +\ b_2\ +\ b_3,$
\end{minipage}\end{equation}
such that
\begin{equation}\begin{minipage}[c]{35pc}\label{d.oAp}
for every prime $p,$ at least one of the decompositions
of~\textup{(\ref{d.aaabbb})} is $\!p\!$-acceptable.
\end{minipage}\end{equation}

Then $N\ \geq\ 1726\ =\ 2^6\cdot 3^3\,-\,2.$

If $N$ is even, then in fact $N\ \geq\ 6910\ =\ 2^8\cdot 3^3\,-\,2.$

In either case, $N{-}\<1$ is divisible by at least $3$ distinct primes.
\end{proposition}

\begin{proof}
From the discussion in the last section, we see that
for every prime $p$ dividing $N$ or $N{-}\<1,$ one of the last
two decompositions of~(\ref{d.aaabbb}) must be $\!p\!$-acceptable.
Thus, if $N$ is divisible by $p^d,$ then
looking at the last $d$ digits of the base-$\!p\!$ expression of $N$
in the light of Definition~\ref{D.accept}, we see that
$p^d$ must divide all three summands in one of those two decompositions;
i.e., either all three of $a_1,\,a_2,\,a_3$ or all
three of $b_1,\,b_2,\,b_3.$
Hence, $N^3\,|\,a_1\,a_2\,a_3\,b_1\,b_2\,b_3.$
In the same way, we see that each prime power factor of $N{-}\<1$ must
either divide two of $a_1,\,a_2,\,a_3$ or two of $b_1,\,b_2,\,b_3,$
hence $(N{-}\<1)^2\,|\,a_1\,a_2\,a_3\,b_1\,b_2\,b_3.$
Since $N$ and $N{-}\<1$ are relatively prime, this gives
\begin{equation}\begin{minipage}[c]{35pc}\label{d.N^3(N-1)^2}
$N^3(N{-}\<1)^2\ |\ a_1\,a_2\,a_3\,b_1\,b_2\,b_3.$
\end{minipage}\end{equation}

On the other hand, it is easy to verify that for any positive real
number, the decomposition into three nonnegative summands having the
largest product is the one in which each summand is one third of the
total; so $a_1\,a_2\,a_3\leq(N/3)(N/3)(N/3);$ and likewise for the
$b_i:$

\begin{equation}\begin{minipage}[c]{35pc}\label{d.N^3/27}
$a_1\,a_2\,a_3\ \leq\ N^3/\,3^3,\qquad b_1\,b_2\,b_3\ \leq\ N^3/\,3^3.$
\end{minipage}\end{equation}

At this point, we could combine~(\ref{d.N^3(N-1)^2})
and~(\ref{d.N^3/27}) to get a lower bound on $N.$
But let us first strengthen each of~(\ref{d.N^3(N-1)^2})
and~(\ref{d.N^3/27}) a bit, using
some special considerations involving the prime $2.$
That prime necessarily divides $N(N{-}\<1);$ assume
without loss of generality that $a_1+a_2+a_3$ is $\!2\!$-acceptable.
Then by~(\ref{d.2vals}), the powers of $2$
dividing $a_1,\ a_2$ and $a_3$ are distinct.
If $2\,|\,N,$ occurring, say, to the $\!d\!$-th power,
this means that $a_1\,a_2\,a_3$ must be divisible not merely by
the factor $2^d\,2^d\,2^d=2^{3d}$ implicit in our derivation
of~(\ref{d.N^3(N-1)^2}), but by $2^d\,2^{d+1}\,2^{d+2}=2^{3d+3}.$
If, rather $2\,|\,N{-}\<1,$
again, say, to the $\!d\!$-th power, we can merely
say that $a_1,\ a_2$ and $a_3$ include, along with one odd term,
terms divisible by $2^d$ and $2^{d+1},$ giving a divisor $2^{2d+1}$ in
place of the $2^{2d}$ implicit in~(\ref{d.N^3(N-1)^2}).
Thus, we can improve~(\ref{d.N^3(N-1)^2}) to
\begin{equation}\begin{minipage}[c]{35pc}\label{d.more_2}
$2\,N^3(N{-}\<1)^2\ |\ a_1\,a_2\,a_3\,b_1\,b_2\,b_3,$\quad and if
$N$ is even,\quad
$8\,N^3(N{-}\<1)^2\ |\ a_1\,a_2\,a_3\,b_1\,b_2\,b_3.$
\end{minipage}\end{equation}

To improve~(\ref{d.N^3/27}), on the other
hand, consider three real numbers
\begin{equation}\begin{minipage}[c]{35pc}\label{d.r1r2r3}
$r_1\ \geq\ r_2\ \geq\ r_3$
\end{minipage}\end{equation}
(which we shall assume given with specified
base-$\!2\!$ expressions, so that, e.g., $1.000\dots$ and $0.111\dots$
are, for the purposes of this discussion, distinct),
subject to the condition that
\begin{equation}\begin{minipage}[c]{35pc}\label{d.nocarry}
there is no need to carry when $r_1,\,r_2,\,r_3$ are added in base~$2;$
\end{minipage}\end{equation}
and suppose we want to know how large the number
\begin{equation}\begin{minipage}[c]{35pc}\label{d.ratio}
$r_1\,r_2\,r_3\,/\,(r_1+r_2+r_3)^3$
\end{minipage}\end{equation}
(regarded as a real number, without distinguishing between
alternative base-$\!2\!$ expansions if these exist) can be.
Note that if, in one of the $r_i,$ we change a base-$\!2\!$ digit $1,$
other than the highest such digit, to $0,$ i.e., subtract $2^d$ for
appropriate $d,$ and simultaneously, for some $j>i$
(cf.~(\ref{d.r1r2r3}))
add $2^d$ to $r_j$ (change the corresponding
digit, which was $0$ by~(\ref{d.nocarry}), to $1),$ then,
proportionately, the decrease in $r_i$ will be less than the increase
in $r_j.$
From this it is easy to deduce that for fixed $r_1+r_2+r_3,$ we
will get the largest possible value for~(\ref{d.ratio}) by
letting $r_1$ contain only the largest digit $1$ of
that sum, $r_2$ only the second largest, and $r_3$ everything else.
(To make this argument rigorous, we have to know that~(\ref{d.ratio})
assumes a largest value.
We can show this by regarding the set of $\!3\!$-tuples of strings
of $\!1\!$'s and $\!0\!$'s, with a
specified number of these to the left of
the decimal point and the rest to the right, and with no two
members of our $\!3\!$-tuple having $\!1\!$'s in the same position,
as a compact topological space under the product topology.
We can then interpret~(\ref{d.ratio}) as a continuous real-valued
function on that space, and conclude that it attains a maximum.)

More subtly, I claim that if some digit of $r_1+r_2+r_3$ after the
leading $1$ is $0,$ then the value~(\ref{d.ratio}) will be increased on
replacing that digit by $1$ in $r_3,$ and hence in $r_1+r_2+r_3.$
This can be deduced using the fact that the partial derivative
of~(\ref{d.ratio}) with respect to $r_3$ is positive, together
with some ad hoc considerations in the case where the digit
in question has value $>r_2.$
Since~(\ref{d.ratio}) is invariant under multiplication of
all $r_i$ by a common power of $2$ (``shifting the decimal
point''), it is not hard to deduce that it is
maximized when $r_1,\ r_2$ and $r_3$
have base-$\!2\!$ expansions $1_2,\ 0.1_2$ and $0.0111\dots\<_2\,.$
In that case, its value is $(1\cdot 1/2\cdot 1/2)/(1+1/2+1/2)^3=2^{-5}.$
Thus we can improve the first inequality of~(\ref{d.N^3/27}) to
\begin{equation}\begin{minipage}[c]{35pc}\label{d.N^3/32}
$a_1\,a_2\,a_3\ \leq\ N^3/\,2^5.$
\end{minipage}\end{equation}
(We cannot similarly improve the second inequality, since
the decomposition $b_1+b_2+b_3$ need not be $\!2\!$-acceptable,
i.e., need not satisfy the analog of~(\ref{d.nocarry}).)

If we now combine~(\ref{d.N^3/27}), so improved, with
the first assertion of~(\ref{d.more_2}), we get
\begin{equation}\begin{minipage}[c]{35pc}\label{d.N...leq}
$2\,N^3(N{-}\<1)^2\ \leq\ (N^3/2^5)(N^3/3^3).$
\end{minipage}\end{equation}
So
\begin{displaymath}\begin{minipage}[c]{33pc}
$(N{-}\<1)^2\ \leq\ N^3/\,(2^6\cdot 3^3),$\quad so\\[.5em]
$N^3/(N{-}\<1)^2\ \geq\ 2^6\cdot 3^3.$
\end{minipage}\end{displaymath}

Expanding the left-hand side in powers of $N{-}\<1,$ we get
$(N{-}\<1)\ +\ 3,$ plus terms whose
sum becomes $<1$ as soon as $N\geq 5.$
Since the above inequality certainly cannot be satisfied by
any integer $N$ with $1<N<5,$ we can discard those terms, getting
\begin{displaymath}\begin{minipage}[c]{33pc}
$N+\,2\ \geq\ 2^6\cdot 3^3.$
\end{minipage}\end{displaymath}

This gives the first assertion of the proposition.
When $N$ is even, we use the second inequality of~(\ref{d.more_2})
in place of the first, getting the second assertion.

To prove the final assertion, suppose $N{-}\<1$ were the product of
two prime powers.
(For brevity, we consider this to include the case where
$N{-}\<1$ is itself a prime power, putting in
a dummy second prime power $1.)$
Then each of these would either divide
two summands in the decomposition $N=a_1+a_2+a_3,$ or
two summands in the decomposition $N=b_1+b_2+b_3.$
If they each divided two summands in the same decomposition,
then they would both divide at least one of those
summands, making that summand $\geq N{-}\<1,$
which is impossible if the three terms are to sum to $N.$
So instead, one prime-power divisor of $N{-}\<1,$ which we shall write
$r,$ must divide two terms of $a_1+a_2+a_3,$ and the other, which
we shall write $s,$ must divide two terms of $b_1+b_2+b_3.$
Moreover, since every prime power dividing $N$ divides
either all of $a_1,\ a_2,\ a_3$ or all of $b_1,\ b_2,\ b_3,$
we can write $N=tu$ where $t$ divides all of the former
and $u$ divides all of the latter.
Now since $a_1+a_2+a_3$ has all terms divisible by $t$ and
at least two divisible by $r,$ and sums to $N,$ we
have $N>2rt;$ and similarly we have $N>2su.$
Multiplying, we get $N^2>4\,r\,s\,t\,u=4N(N{-}\<1),$ which is
impossible.
\end{proof}

Remark: If one tries to extend the argument of the last paragraph
above to the case where $N{-}\<1$ is a product of three prime powers,
one discovers one case which there is no obvious way to exclude:
One of our given decompositions, say $N=a_1+a_2+a_3,$ might
be $\!p\!$-acceptable for all three of those primes,
with one prime power factor dividing $a_1$ and $a_2,$
another dividing $a_1$ and $a_3,$
and the third dividing $a_2$ and $a_3.$
The product $t$ of those prime power factors of $N$ that divide all of
$a_1,\ a_2,\ a_3$ could then be nontrivial, though it would have to
be smaller than each of the three prime power factors of $N{-}\<1.$

\section{The case $i+j>2.$}\label{S.i+j>2}

Let us now consider a possible counterexample to Conjecture~\ref{Cj} for
$k=3$ of the contrary sort, where the member of our
triad of decompositions involving a summand $\geq p_{\max}^d,$
\begin{equation}\begin{minipage}[c]{35pc}\label{d.N-i-j_rep}
$N\ =\ (N{-}\<i\<{-}\<j)+i+j$
\end{minipage}\end{equation}
has at least one of the remaining summands, $i$ and $j,$
greater than $1.$
We will find it most convenient to begin by assuming
we are given only the other two decompositions in our triad,
\begin{equation}\begin{minipage}[c]{35pc}\label{d.aaabbb2}
$N\ =\ a_1\ +\ a_2\ +\ a_3,\qquad N\ =\ b_1\ +\ b_2\ +\ b_3,$
\end{minipage}\end{equation}
and develop results on such a pair of decompositions, from
which we will subsequently obtain constraints on the
decomposition~(\ref{d.N-i-j_rep}).

The primes dividing $N(N{-}\<1)(N{-}\<2)$ that will eventually
have to be taken care of by the decomposition~(\ref{d.N-i-j_rep})
are given a name in
\begin{definition}\label{D.relevant}
Given two decompositions\textup{~(\ref{d.aaabbb2})} of an
integer $N$ into positive integer summands, we shall
call a prime $p$ {\em relevant}
\textup{(}with respect to\textup{~(\ref{d.aaabbb2}))} if it divides
$N(N{-}\<1)(N{-}\<2),$ but neither of the
decompositions\textup{~(\ref{d.aaabbb2})} is $\!p\!$-acceptable.

If $p$ is a relevant prime, and $p^d$ $(d>0)$ a prime power factor
\textup{(}in the sense defined at the beginning of
the preceding section\textup{)}
of any of $N,\ N{-}\<1$ or $N{-}\<2,$ we will call $p^d$ a
{\em relevant prime power.}
\end{definition}

To measure the impact of the relevant primes in the estimates to
be made, we define
\begin{equation}\begin{minipage}[c]{35pc}\label{d.C}
$C\ =$ the product of all relevant prime power factors of $N{-}\<2,$
the squares of all relevant prime power factors of $N{-}\<1,$
and the cubes of all relevant prime power factors of $N.$
\end{minipage}\end{equation}
(Note that if $2$ is a relevant prime and $N$ is even, then the
computation of $C$ will involve both the power of $2$ dividing
$N{-}\<2,$ and the cube of the power of $2$ dividing $N.)$

The next result, using ideas similar to those of the
preceding section, shows that $C$ must be fairly large.

\begin{lemma}\label{L.C}
In the situation of\textup{~(\ref{d.aaabbb2})}
and~\textup{(\ref{d.C})}, one always has
\begin{equation}\begin{minipage}[c]{35pc}\label{d.C>3^6}
$C\ >\ 3^6\,(1-4\,N^{-1}).$
\end{minipage}\end{equation}
\textup{(}In particular if $N\geq 12,$ then $C>486,$ and
if $N\geq 81,$ then $C>693.)$

If, moreover, $2$ is not a relevant prime \textup{(}with
respect to\textup{~(\ref{d.aaabbb2}))}, then
\begin{equation}\begin{minipage}[c]{35pc}\label{d.C>3^3_2^6}
$C\ >\ 3^3\cdot 2^6\,(1-4\,N^{-1}).$
\end{minipage}\end{equation}
\end{lemma}

\begin{proof}
Let us write $C_i$ for the product of the
relevant prime powers dividing $N{-}\<i$ $(i=0,\,1,\,2).$
Then we see (from the definition of $\!p\!$-acceptability)
that each prime power dividing $N/C_0$ will either divide
all of $a_1,\,a_2,\,a_3$ or all of $b_1,\,b_2,\,b_3,$
each prime power dividing $(N{-}\<1)/C_1$ will either divide
two of $a_1,\,a_2,\,a_3$ or two of $b_1,\,b_2,\,b_3,$ and
each prime power dividing $(N{-}\<2)/C_2$ will either divide
at least one of $a_1,\,a_2,\,a_3$ or at least one of $b_1,\,b_2,\,b_3.$
Hence $a_1\,a_2\,a_3\,b_1\,b_2\,b_3$ will be divisible by
$(N{-}\<2)/C_2,$ by $(N{-}\<1)^2/C_1^2,$ and by $N^3/C_0^3.$
Moreover, the only prime that can divide more than one of three
successive integers is $2,$ so $a_1\,a_2\,a_3\,b_1\,b_2\,b_3$
will be divisible by the product of these integers,
$N^3(N{-}\<1)^2(N{-}\<2)/C,$ possibly
adjusted by an appropriate power of $2.$

That adjustment will not be needed if $2$ is a relevant prime,
since in that case, by definition
of the $C_i,$ those remove all divisors $2$ from our expression.
It also will not come in if $N$ is odd,
since then only one of $N{-}\<2,\ N{-}\<1,\ N,$ namely $N{-}\<1,$
is divisible by $2.$
So in both of those cases we have
\begin{equation}\begin{minipage}[c]{35pc}\label{d.N^3(N-1)^2(N-2)}
$N^3(N{-}\<1)^2(N{-}\<2)\,/\,C\ |\ a_1\,a_2\,a_3\,b_1\,b_2\,b_3.$
\end{minipage}\end{equation}
Combining this with~(\ref{d.N^3/27}), we get, in these cases
\begin{equation}\begin{minipage}[c]{35pc}\label{d.N^6/3^6}
$N^3(N{-}\<1)^2(N{-}\<2)\,/\,C\ \leq\ N^6/\,3^6,$
\end{minipage}\end{equation}
that is,
\begin{equation}\begin{minipage}[c]{35pc}\label{d.Cgeq}
$C\ \geq\ 3^6\,(1-N^{-1})^2(1-2\,N^{-1}).$
\end{minipage}\end{equation}
(We cannot improve~(\ref{d.N^3/27}) using~(\ref{d.N^3/32}) in
this calculation: the argument that previously allowed us to do so
only applies if $2$ is not a relevant prime.)
When we expand the product of the last two factors in~(\ref{d.Cgeq}),
we see that the $N^{-2}$ term has coefficient $+5$ and the $N^{-3}$ term
coefficient $-2,$ so their sum is positive for all $N\geq 1,$
and we may drop those terms, getting~(\ref{d.C>3^6}) under these
conditions; in particular, whenever $2$ is a relevant prime.

We now consider the case where $2$ is not a relevant prime.
Say the decomposition $a_1+a_2+a_3$ is $\!2\!$-acceptable.
Then we can, as in the proof of Proposition~\ref{P.(N-2)+1+1},
replace the first inequality of~(\ref{d.N^3/27}) by~(\ref{d.N^3/32}),
and so improve our upper bound on $a_1\,a_2\,a_3\,b_1\,b_2\,b_3$
from~$N^6/3^6$ to~$N^6/3^3\cdot 2^5.$

Now if $N$ is odd, we have noted that we still have~(\ref{d.N^6/3^6});
moreover, as in the proof of Proposition~\ref{P.(N-2)+1+1}, we can
use~(\ref{d.2vals}) to put a
factor of $2$ on the left side of that inequality
(since if $N{-}\<1$ is divisible by $2^d,$ then
$a_1\,a_2\,a_3$ will be divisible by $2^d\,2^{d+1}).$
Combining this with the modification of the right-hand
side indicated in the preceding paragraph,
we get~(\ref{d.C>3^3_2^6}) for such $N.$

When $N$ is even, we must analyze more closely the relation between
the powers of $2$ dividing $N^3(N{-}1)^2(N{-}\<2)$ and $a_1\,a_2\,a_3.$
Exactly one of $N$ and $N{-}\<2$ will be divisible by $2^d$
for some $d>1.$
Assume first that $2^d\,|\,N.$
Then the power of $2$ dividing $N^3(N{-}1)^2(N{-}\<2)$ is
$2^{3d}\cdot 2^0\cdot 2^1=2^{3d+1},$
while the power dividing $a_1\,a_2\,a_3$
will be at least $2^d\cdot 2^{d+1}\cdot 2^{d+2}=2^{3d+3},$
giving {\em two} extra powers of $2,$ and hence an inequality
that is in fact stronger than~(\ref{d.C>3^3_2^6}).
If, rather, $2^d\,|\,N{-}\<2,$ then the power
of $2$ dividing $N^3(N{-}1)^2(N{-}\<2)$ will
be $2^3\cdot 2^0\cdot 2^d=2^{d+3},$
while the power dividing $a_1\,a_2\,a_3$
will be at least $2^1\cdot 2^d\cdot 2^{d+1}=2^{2d+2}\geq 2^{d+4},$
which provides a single extra factor of $2,$
and so again gives~(\ref{d.C>3^3_2^6}).
\end{proof}

We now apply the above to a possible counterexample to
Conjecture~\ref{Cj}.

\begin{proposition}\label{P.leq10}
Suppose a positive integer $N$ admits three decompositions, which we
will write~\textup{(\ref{d.N-i-j_rep})} and~\textup{(\ref{d.aaabbb2})},
such that for every prime $p$ dividing $N(N{-}\<1)(N{-}\<2),$ at
least one of these decompositions is $\!p\!$-acceptable.
Then we cannot have $2<i+j<11.$
\end{proposition}

\begin{proof}
Below, ``relevant'' will mean relevant
with respect to~(\ref{d.aaabbb2}).
We claim first that
\begin{equation}\begin{minipage}[c]{35pc}\label{d.leq2}
If $p^d$ is a relevant prime power, then there is no carrying when
$N{-}\<i\<{-}\<j,\ i$ and $j$
are added in base~$p,$ and the remainders on dividing
$i$ and $j$ by $p^d$ sum to at most $2.$
\end{minipage}\end{equation}
Indeed, the first assertion follows from our hypothesis, which
implies that the decomposition~(\ref{d.N-i-j_rep})
is $\!p\!$-acceptable for every relevant prime $p.$
Combining this with the fact that a relevant prime power $p^d$ is by
definition a divisor of $N,\ N{-}\<1$ or $N{-}\<2,$
hence that $N\equiv 0,\,1$ or $2\pmod{p^d},$
we get the second assertion.

Let us now, by way of contradiction,
\begin{equation}\begin{minipage}[c]{35pc}\label{d.2<11}
assume that $2<i+j<11.$
\end{minipage}\end{equation}
Then $i$ and $j,$ though positive, are not both $1;$ so for
$p$ as in~(\ref{d.leq2}), we see from the second assertion
thereof that one of $i$ or $j$ must be $\geq p^d.$
Thus,
\begin{equation}\begin{minipage}[c]{35pc}\label{d.p^d+1}
if $p^d$ is a relevant prime power, $i+j\ \geq\ p^d+1.$
\end{minipage}\end{equation}
This and the upper bound of~(\ref{d.2<11}) limit the
possible relevant prime powers to
\begin{equation}\begin{minipage}[c]{35pc}\label{d.2-9}
$2,\ 3,\ 4,\ 5,\ 7,\ 8,\ 9.$
\end{minipage}\end{equation}

Now the hypotheses of the proposition cannot be satisfied for
any $N\leq 14.$
Indeed, for each such value, $N$ or $N{-}\<1$
is a prime power $p^d,$ hence $N$ has digit-sum $\leq 2$
to base $p,$ so for such $p$ there can be no $\!p\!$-acceptable
decomposition of $N$ (cf.~(\ref{d.sum<k}),~(\ref{d.25_<100})).
Hence we may assume $N>14,$ and so use the bound
\begin{equation}\begin{minipage}[c]{35pc}\label{d.486}
$C>486$
\end{minipage}\end{equation}
of Lemma~\ref{L.C}.

Let us now consider a relevant prime
power $p_0^d$ dividing $N$ itself.
In that case,~(\ref{d.leq2}) implies
that $p_0^d$ must divide both $i$ and $j,$ so~(\ref{d.2<11})
limits us to $p_0^d=2,\ 3,\ 4,\ 5.$
Of these values, $4$ is excluded, since the condition that there be no
carrying when $i$ and $j$ are added in base $p=2$ shows that if
$i,\ j$ are divisible by $4,$ one of them is divisible by $8,$
making their sum at least $12;$ we are left with $p_0^d=2,\ 3,\ 5.$
If $p_0^d=3$ or $5,$ then the second statement of~(\ref{d.leq2}),
together with the divisibility of $i$ and $j$ by $p_0^d,$
and~(\ref{d.2<11}), force $i=j=p_0.$
But this has the consequence that there can
be no relevant prime
(whether dividing $N,\ N{-}\<1$ or $N{-}\<2)$
{\em other} than $p_0:$ indeed, for these choices
of $i$ and $j,$ each of the other prime powers in the list~(\ref{d.2-9})
is eliminated by at least one of the conditions of~(\ref{d.leq2}).
This gives $C=p_0^3\leq 125,$ contradicting~(\ref{d.486}).
Likewise, if $p_0^d=2,$ then the combination
of the first condition of~(\ref{d.leq2})
and our bound on $i+j$ excludes all possibilities but
$\{i,j\}=\{2,4\}$ and $\{2,8\}.$
Neither of these choices is consistent
with~(\ref{d.leq2}) holding for any of our odd prime powers.
So in this case $C\leq 8\cdot 2^3=64,$ again
contradicting~(\ref{d.486}).
These contradictions show that no relevant primes divide $N.$

Knowing this, let us again look through the possible relevant primes.
If $7$ is relevant, then by~(\ref{d.leq2}) and~(\ref{d.2<11}),
the unordered pair
$\{i,j\}$ must be one of $\{7,1\},$ $\{7,2\}$ or $\{8,1\}.$
Applying~(\ref{d.leq2}) to the other prime powers in~(\ref{d.2-9}),
we see that for $\{i,j\}=\{7,1\},$ the only other relevant prime power
could be a $3,$ in which case the final digit of $N$ to base $3$ would
be $2,$ so that $3$ would divide $N{-}\<2,$ making
$C\leq 3\cdot 7^2=147.$
For $\{i,j\}=\{7,2\},$ there are no other possible relevant prime
powers; and
for $\{i,j\}=\{8,1\},$ we could at most have an $8$ dividing $N{-}\<1,$
which, since $7$ would divide $N{-}\<2,$ gives $C\leq 7\cdot 8^2=448;$
in each case contradicting~(\ref{d.486}).
If $5$ is relevant, the possibilities for $\{i,j\}$ are
$\{5,1\},$ $\{5,2\},$ $\{6,1\}$ and $\{5,5\}.$
Looking at the cases where $i$ and/or $j$ is $5,$ we see that the
only other possible relevant prime power
consistent with~(\ref{d.leq2}) is $2,$ which could divide
$N{-}\<1$ in the one case $\{5,2\},$ giving $C\leq 2^2\cdot 5^2=100;$
while in the case $\{6,1\},$ we get $2$ and $3$ as possible
relevant prime power factors of $N{-}\<1,$ which, together with $5$
dividing $N{-}\<2,$ would give $C\leq 5\cdot 2^2\cdot 3^2=180;$
again each case contradicts~(\ref{d.486}).
So at most $2$ and $3$ can be relevant primes.
If $9$ were a relevant prime power, then
by~(\ref{d.2<11}) we could only
have $\{i,j\}=\{9,1\},$ so $2$ would not be relevant,
and $C$ would be at most $9^2=81.$
Likewise, if $8$ were a relevant prime power, then we could only
have $\{i,j\}=\{8,1\},$ or $\{8,2\},$ so $3$ would not be relevant,
and $C$ would be at most $8^2=64.$
So the relevant prime powers are at most $3$ and $4,$ each
occurring in the expression for $C$ to at most the second
power, giving $C\leq 144,$ and yielding the same contradiction.
\end{proof}

We remark
that the argument just given cannot be extended to exclude $i+j=11.$
Indeed, for $N$ of the form $180\,M+11,$ consider the decomposition
\begin{equation}\begin{minipage}[c]{35pc}\label{d.180}
$N =\ 180\,M\ +\ 10\ +\ 1.$
\end{minipage}\end{equation}
I claim that for infinitely many values of $M,$ the above
decomposition of $N$ is $\!2\!$-, $\!3\!$- and $\!5\!$-acceptable,
with $2\cdot 5\,|\,N{-}\<1$ and $3^2\,|\,N{-}\<2.$
Indeed, we see that the conditions for $\!2\!$-, $\!3\!$- and
$\!5\!$-acceptability of~(\ref{d.180}) are that the base-$\!2\!$
expansion of $180\,M$ have $\!2^3\!$-digit zero,
that its base-$\!3\!$ expansion have $3^2$ digit $0$ or $1,$
and that its base-$\!5\!$ expansion have $5^1$ digit $0,\ 1$ or $2;$
so these are satisfied $(1/2)\,(2/3)\,(3/5)=1/5$ of the time.
(The first few values satisfying these conditions
are $M=5,\,12,\,17.)$
Thus, if we combine~(\ref{d.180}) with two other decompositions
of $N,$ the latter need not be $\!2\!$-, $\!3\!$- or $\!5\!$-acceptable,
so we can get $C=2^2\cdot 3^2\cdot 5^2=900,$ which no longer
contradicts~(\ref{d.C>3^6}).
(Here $2$ {\em is} a relevant prime, so we cannot
use the stronger conclusion~(\ref{d.C>3^3_2^6}).)
I suspect that as we take still larger values of $i+j,$ we can get
arbitrarily large $C.$

However, the bound of Proposition~\ref{P.leq10} is enough
to eliminate a large range of values of $N,$ as we shall now show.

\section{There are no counterexamples with $k=3,\ N<785.$}\label{S.785}

Propositions~\ref{P.(N-2)+1+1} and~\ref{P.leq10} together show us that
in looking for counterexamples to Conjecture~\ref{Cj} with $k=3$
and $N<1726,$ it suffices to check values of $N$
which exceed by at least $11$ the largest prime power $\leq N.$
In particular, $N$ must lie in a gap of length at least $12$ between
successive prime powers.
A search through a list of primes shows $17$ gaps of length $\geq 12$
with the lower prime less than $1000,$
their lengths ranging from $12$ to $20.$
Several of these are thrown out of the picture when we bring
higher prime powers into consideration.
(The length-$\!14\!$ gap between $113$ and $127$
is interrupted by $121$ and $125,$
the one between $953$ and $967$ by $31^2=961,$ and
the length-$\!12\!$ gaps between $509$ and $521,$
and between $619$ and $631,$ by $2^9=512$ and $5^4=625$ respectively.)
A couple of other gaps are shortened from greater lengths down
to length $12$ in this way.
(The one between $523$ and $541$ by $23^2=529,$
and the one between $839$ and $853$ by $29^2=841.)$
The surviving gaps, with those of length $>12$ shown in boldface, are
\begin{equation}\begin{minipage}[c]{35pc}\label{d.gaps}
$(199,211),$ $(211,223),$ $\mathbf{(293,307)},$
$\mathbf{(317,331)},$ $(467,479),$ $(529,541),$ $(661,673),$\\
$\mathbf{(773,787)},$ $(797,809),$ $(841,853),$ $\mathbf{(863,877)},$
$\mathbf{(887,907)},$ $(997,1009).$
\end{minipage}\end{equation}

Since most of these gaps have length $12,$ a large fraction of
the values of $N$ that this list informs us are not covered
by Proposition~\ref{P.leq10} are of the form $N=p_{\max}^d+11.$
Many such cases, including all in the above
list, can be eliminated using the following result.

\begin{lemma}\label{L.relpm}
Suppose~\textup{(\ref{d.N-i-j_rep})}
and~\textup{(\ref{d.aaabbb2})} are decompositions of $N$
such that for every prime $p$ dividing $N(N{-}\<1)(N{-}\<2),$ at
least one of these decompositions is $\!p\!$-acceptable.
Suppose moreover that in~\textup{(\ref{d.N-i-j_rep})},
$i$~and $j$ are relatively prime to one another, and
$N{-}\<i{-}\<j$ is relatively prime to $i+j{-}\<1.$

Then no divisor of
$N$ or $N{-}\<1$ is a relevant prime; hence $C\ |\ N{-}\<2.$
\end{lemma}

\begin{proof}
Any relevant prime $p$ dividing $N$ would have to divide
all of $N{-}\<i{-}\<j,$ $i$ and $j,$ which is excluded by
the relative primality of the last two of these,
while a relevant prime $p$ dividing $N{-}\<1$ must divide two of
$N{-}\<i{-}\<j,$ $i$ and $j,$ with the remaining one being
$\equiv 1~(\r{mod}\ p).$
Each choice of which two terms are $\equiv 0$ modulo $p$
and which is $\equiv 1$ is excluded by
one or the other of our relative primality hypotheses.
\end{proof}

Now if $N=p_{\max}^d+11<1009$ is a counterexample to
Conjecture~\ref{Cj}, with $\!p_{\max}\!$-admissible
decomposition~(\ref{d.N-i-j_rep}), then Propositions~\ref{P.(N-2)+1+1}
and~\ref{P.leq10} exclude all possible values for $i+j$ other than $11.$
Since $i+j=11$ is prime, $i$ and $j$ must be relatively prime;
if, moreover, $p_{\max}^d$ is not a power of $2$ or $5$ (as
indeed none of the first members of the pairs in~(\ref{d.gaps}) are),
it must be relatively prime to $i+j-1=10.$
So Lemma~\ref{L.relpm} tells us that in these cases,
$C\mid N{-}\<2,$ so $C\leq 1009-2.$
Lemma~\ref{L.relpm} also shows that $2\ |\ N(N{-}\<1)$ is
not a relevant prime, so~(\ref{d.C>3^3_2^6}) says that
$C>1728\cdot(1-4\cdot 200^{-1})>1693$ (since the values of $N$ arising
from~(\ref{d.gaps}) are all $> 200),$ contradicting the
preceding sentence, and eliminating these cases.

The values of $N<1009$ not eliminated by this argument are those
that exceed the greatest prime power $\leq$ them by at least $12.$
From~(\ref{d.gaps}), these are
\begin{equation}\begin{minipage}[c]{35pc}\label{d.329...}
$305,\ 306;\ \ 329,\ 330;\ \ 785,\ 786;\ \ 875,\ 876;$ \ and
\  $899,\,\cdots,\,906\ (8\!$ terms).
\end{minipage}\end{equation}

Feeling that case-by-case elimination of possible $N$ is
not a way I want to continue to pursue this problem, I have
only checked the first four of these.
By looking at the primes dividing $N,\ N{-}\<1$ and $N{-}\<2$
in these cases, it is not hard to find properties that
exclude each of these values of $N.$
This is done in the proposition below.
Let us start with a definition and two lemmas that formalize a
kind of observation that we will use.
(Note that the ``digit-sum $\geq k$'' condition in the next
definition is simply the necessary and
sufficient condition for there to be {\em any} $\!p\!$-acceptable
decompositions of $N.)$

\begin{definition}\label{D.thresh}
If $k$ and $N$ are positive integers, and $p$ a prime such
that the base-$\!p\!$ expression for $N$ has digit-sum
$\geq k,$ then by the {\em $\!p\!$-threshold} for $N$ \textup{(}with
respect to $k)$ we shall mean the greatest integer $m$ occurring
in any $\!p\!$-acceptable expression for $N$ as a sum of $k$
positive integers.
\end{definition}

From the characterization of $\!p\!$-acceptable decompositions
at the end of Definition~\ref{D.accept}, we see

\begin{lemma}\label{L.thresh}
For $k=3,$ and $N$ and $p$ as
in Definition~\ref{D.thresh}, the $\!p\!$-threshold for $N$
is $N-p^{e_1}-p^{e_2},$ where $p^{e_1}$ is the largest
power of $p$ dividing $N,$ and $p^{e_2}$ is the largest
power of $p$ dividing $N-p^{e_1}.$

\textup{(}For general $k,$ one has the corresponding description with
$k-1$ successive subtractions.\textup{)}\qed
\end{lemma}

\begin{lemma}\label{L.digsum5}
Let $N$ be a positive integer, and
$p_{\max}^d$ the largest prime power $\leq N.$
Suppose that for some prime $p_0\neq p_{\max}$ and not dividing $N,$
the base-$\!p_0\!$ expression for $N$ has digit-sum~$\leq 5,$ and
that for all primes $p$ dividing $N,$ and also for $p=p_0,$ the
$\!p\!$-threshold for $N$ with respect to $k=3$ is $<p_{\max}^d.$

Then $N$ is not a counterexample to Conjecture~\ref{Cj} for $k=3.$
\end{lemma}

\begin{proof}
Suppose $N$ were a counterexample, with
decompositions~(\ref{d.N-i-j_rep}),~(\ref{d.aaabbb2}),
where the first is $\!p_{\max}$-acceptable.
By our hypothesis on $\!p\!$-thresholds,
neither $p_0$ nor any of the primes dividing $N$ can
be relevant primes; so for each of these, one of the
decompositions of~(\ref{d.aaabbb2}) must be $\!p\!$-acceptable.
Let $a_0+a_1+a_2$ be the $\!p_0\!$-acceptable decomposition.

Since in base $p_0,$ $N$ has digit-sum $<6,$
at least one of $a_0,\,a_1,\,a_2$ must have digit-sum $<2,$
i.e., must be a power of $p_0.$
This term will be relatively prime to all the primes
dividing $N,$ so for each of those primes $p,$ the
$\!p\!$-acceptable decomposition must be $b_0+b_1+b_2.$
This makes each of $b_1,\,b_2,\,b_3$ a multiple of $N,$
contradicting the assumption that they sum to $N.$
\end{proof}

We can now verify

\begin{proposition}\label{P.304}
There are no counterexamples to Conjecture~\ref{Cj}
for $k=3$ with $N<785.$
\end{proposition}

\begin{proof}
Examining~(\ref{d.329...}), we see that we
must check $N=305,\ 306,\ 329$ and $330.$
From~(\ref{d.gaps}) we see that for the first
two of these, $p_{\max}^d=293,$ while for the last two it is $317.$

The case $306$ is excluded by~(\ref{d.sum<k}), since its
base-$\!17\!$ expansion is $110_{17},$
while $305$ is excluded by Lemma~\ref{L.digsum5} with $p_0=2.$

The case $329,$ which has base-$\!2\!$ expression $101001001_2,$
would likewise be excluded by that lemma with $p_0=2,$ except
that its $\!2\!$-threshold is $320,$ which is not $<317.$
However, in the proof
of that lemma, the condition that the $\!p_0\!$-threshold of
$N$ be $<p_{\max}^d$ is used only to rule out the possibility that
$p_0$ is a relevant prime.
Now if $2$ were a relevant prime, we would have
$i+j\leq N-p_{\max}= 329-317=14,$ and we see from the above
base-$\!2\!$ expression that $i+j$ would have to be $8+1=9.$
This is ruled out by Proposition~\ref{P.leq10}; so $329$ is excluded.

To handle $330,$ note that again $p_{\max}=317,$ and that now
$330=101001010_2.$
Here we will use an argument similar to that
of Lemma~\ref{L.digsum5}, but with the roles
of $N$ and $N{-}\<1$ reversed.
Essentially the same reasoning as in the preceding case
shows that $2$ cannot be a relevant prime.
Moreover, since the base-$\!2\!$ expression for $N$ has
digit-sum $4,$ any $\!2\!$-acceptable decomposition of $N$ must have
two terms that are powers of $2;$ hence such a
decomposition cannot be $\!p\!$-acceptable for any $p\ |\ N{-}\<1.$
Looking at the prime factorization of $N{-}\<1=329=47\cdot 7,$
we see that $47$ has threshold $<317,$ while
$7$ cannot be relevant by essentially the same reasoning we
applied to $2;$ so any three decompositions of $330$ comprising a
counterexample to Conjecture~\ref{Cj} must consist of
a $\!p_{\max}\!$-acceptable decomposition, a $\!2\!$-acceptable
decomposition, and a decomposition acceptable for both factors
of $N{-}\<1.$
This forces at least one summand in the last of these decompositions
to be divisible by the product of those factors, $N{-}\<1,$ making
it too large.
\end{proof}

\section{Quick counterexamples to some plausible
strengthenings of Conjecture~\ref{Cj}.}\label{S.ceg}

It is natural to ask whether Conjecture~\ref{Cj} is the
``right'' statement, or whether some stronger statement might hold.
For instance, for $k>2,$ might the number of proper
$\!k\!$-nomial coefficients that
are guaranteed to have a common divisor grow faster than $k$?

An example showing that {\em four}
proper {\em trinomial} coefficients of the same
weight need not have a common divisor is given by the following four
decompositions of $159.$
\begin{equation}\begin{minipage}[c]{35pc}\label{d.159}
$157+1+1,\qquad 144+12+3,\qquad 53+53+53,\qquad 79+79+1.$
\end{minipage}\end{equation}
(The only primes not handled by the first decomposition are
the divisors of $159\,=\,53\cdot 3$ and $158\,=\,79\cdot 2.$
Of these, $2$ and $3$ are handled by the second decomposition,
and the remaining two primes by the last two.)

Might the importance of the condition that our multinomial
coefficients be proper and of nomiality $k$ simply
be to insure that all their arguments are $\leq N{-}\<k\<{+}\<1$?
If so, we would expect binomial coefficients with all
arguments strictly greater than $1$ to behave as well as is
conjectured for trinomial coefficients.
A counterexample is given by the following three decompositions of $46.$
\begin{equation}\begin{minipage}[c]{35pc}\label{d.46}
$44+2,\qquad 36+10,\qquad 23+23.$
\end{minipage}\end{equation}

Finally, might it be possible to
replace the assumption of $k$ multinomial coefficients, of equal weight
and common nomiality $k,$ with that of a finite family of multinomial
coefficients of equal weights but possibly varying nomialities,
such that the sum of the reciprocals of those nomialities is~$\leq 1$?
Here a counterexample is given by the following three
decompositions of $65,$
\begin{equation}\begin{minipage}[c]{35pc}\label{d.65}
$64+1,\qquad 25+25+5+5+5,\qquad 13+13+13+13+13,$
\end{minipage}\end{equation}
where the reciprocals of the nomialities sum
to $1/2\,+\,1/5\,+\,1/5\,=\,9/10<1.$
(If one wants the sum to be exactly $1,$
one can replace any two terms of the second decomposition by
their sum, and similarly in the third decomposition.)

In \cite[\S4]{from_nomial}, it is shown that, assuming the truth of
Schinzel's Conjecture on prime values assumed
by polynomials with integer coefficients~\cite{SS}, all
these counterexamples belong to infinite families.

\section{Where do we go from here?}\label{S.whither}

A proof of Conjecture~\ref{Cj}, even for $k=3,$ may well require
an approach entirely different from that of
\S\S\ref{S.k=3}-\ref{S.785} above.
On the other hand, if it is false, and we want to find a
counterexample, or if, on the contrary, the approach of this note can
somehow be extended to a proof, our results indicate a bifurcation
of the problem into two cases, the one where the decomposition
with largest summand has the other two summands $i$ and $j$ both
$1,$ and the case where $i$ and $j$ sum to at least $11.$
(Might one get additional mileage by subdividing the
latter case further, according to whether {\em one}
of $i$ and $j$ equals $1$?)

In the case $i+j>2,$ we have made use of primes
dividing $N(N{-}\<1)(N{-}\<2);$
but I suspect that these are not enough -- that we must in
some way also use the facts that
for every prime $p$ dividing $(N{-}\<3)\dots(N{-}\<i\<{-}\<j\<{+}\<1),$
one of our decompositions is $\!p\!$-acceptable.
A suggestion as to how this might be attempted is sketched
in \mbox{\cite[\S5]{from_nomial}}.

I have not tried a computer search for counterexamples.
Someone skilled at such searches might use the
results proved in this note to limit the cases (both the values
of $N$ and the triads of decompositions thereof) to be checked.
(The reference in~\cite[p.131]{RKG} to ``fairly extensive computer
evidence'' was a misunderstanding; all the computations Wasserman and
I did were by hand.
See \cite[\S3]{from_nomial} for an example.)

Some additional restrictions limiting the decompositions
of $N$ that would have to be checked in such searches
are noted in \cite[\S6]{from_nomial}.
The idea is that in the proof of
Proposition~\ref{P.(N-2)+1+1} above, if instead of
the estimate $a_1\,a_2\,a_3\,\leq\,N^3/\,3^3,$ we use
$a_1\,a_2\,a_3\,\leq\,a_1\,N^2/\,2^2,$ we get, instead
of a lower bound on $N,$ a lower bound on $a_1$ independent of $N,$
while if we do the same in the proof of Lemma~\ref{L.C} we
get, instead of a bound on $C$ independent of $N,$ a
bound on $a_1\,C$ that grows linearly in $N.$
The bounds in question say that when $i=j=1,$ all
$a_i$ and $b_i$ are $\geq 216,$ while when $i+j>1,$ all
$a_i\,C$ and $b_i\,C$ are $\geq 108(N\<{-}\<4);$ in each case
with strengthenings under additional assumptions.

Turning to general $k,$ we remark that the equality of the
degrees of the two sides of~(\ref{d.N^6/3^6}) is special to $k=3.$
For large $k,$ the degree of
the left-hand side of the analogous inequality
grows as $k^2/2,$ while that of the right-hand side grows as $k^2;$
so the methods we have been using for $k=3$ are not likely to
extend to larger $k.$
Conjecture~\ref{Cj} might in fact be too strong; perhaps the largest
$k'$ such that every family of $k'$ proper
$\!k\!$-nomial coefficients has a
common divisor satisfies $k'\approx k/\sqrt 2$ when $k$ is large,
rather than $k'=k.$

The hope for a proof of Conjecture~\ref{Cj} (or some
variant) in the spirit of the
group-theoretic proof of Theorem~\ref{T.ESz} is appealing.
Sticking with $k=3$ for simplicity, let us ask

\begin{question}\label{Q.XYZ}
Suppose $G$ is a group, and $X,\ Y,\ Z$ are finite
transitive $\!G\!$-sets, such that
\begin{displaymath}\begin{minipage}[c]{33pc}
$\r{g.c.d.}(\r{card}(X),\ \r{card}(Y),\ \r{card}(Z))\ =\ 1,$
\end{minipage}\end{displaymath}
but such that none of the $\!G\!$-sets $X\times Y,$ $Y\times Z,$
$Z\times X$ is transitive \textup{(}so that
no two of $\r{card}(X),$ $\r{card}(Y),$ $\r{card}(Z)$
are relatively prime\textup{)}.

What, if anything, can one conclude about the orbit-structure
of the $\!G\!$-set $X\times Y\times Z$?
\end{question}

On the other hand,
other interpretations of multinomial coefficients might be
useful in attacking Conjecture~\ref{Cj}.
One result of that sort, conjectured by F.\,Dyson, and proved by
several others (\cite{IJG} and papers cited there) describes
$\ch(a_1,\dots,a_k)$ as the constant term of the Laurent
polynomial $\prod_{i\neq j}(1-x_j/x_i)^{a_i}$ in indeterminates
$x_1,\dots,x_n.$



\begin{thebibliography}{00}
\bibitem{IJG} I. J. Good,
{\it Short proof of a conjecture by Dyson,}
J. Mathematical Phys. {\bf 11} (1970) 1884.
~MR~{\bf 41}\#3290.
(In the second display in this telegraphic note, the final equation
$i=j$ is, as far as I can see, meaningless, and should be ignored.)
\end{thebibliography}
\end{document}